\documentclass[graybox]{svmult}

\usepackage{mathptmx}       
\usepackage{helvet}         
\usepackage{courier}        
\usepackage{type1cm}        
\usepackage{makeidx}         
\usepackage{graphicx}        
\usepackage{multicol}        
\usepackage[bottom]{footmisc}

\usepackage{amsfonts}
\usepackage{eufrak}
\usepackage{xypic}
\usepackage{amssymb}
\usepackage{amsmath}
\usepackage{mathrsfs}

\providecommand{\hpalma}{\mathcal{H}}

\makeindex             

\begin{document}

\title*{Quasi-symmetric group algebras and $C^*$-completions of Hecke algebras}
\author{Rui Palma}
\institute{\emph{Date:} \today\\
Rui Palma \at University of Oslo, P.O. Box 1053, NO-0316 Oslo, Norway, \email{ruip@math.uio.no}\\
Research supported by the Research Council of Norway and the Nordforsk research network ``Operator Algebra and Dynamics''.}

%
%
\maketitle

\abstract{We show that for a Hecke pair $(G, \Gamma)$ the $C^*$-completions $C^*(L^1(G, \Gamma))$ and $pC^*(\overline{G})p$ of its Hecke algebra coincide whenever the group algebra $L^1(\overline{G})$ satisfies a spectral property which we call ``quasi-symmetry'', a property that is satisfied by all Hermitian groups and all groups with subexponential growth. We generalize in this way a result of Kaliszewski, Landstad and Quigg \cite{schl palma}.
Combining this result with our earlier results in \cite{palma} and a theorem of Tzanev \cite{tzanev palma} we establish that the full Hecke $C^*$-algebra exists and coincides with the reduced one for several classes of Hecke pairs, particularly all Hecke pairs $(G, \Gamma)$ where $G$ is nilpotent group. As a consequence, the category equivalence studied by Hall \cite{hall palma} holds for all such Hecke pairs. We also show that the completions $C^*(L^1(G, \Gamma))$ and $pC^*(\overline{G})p$ do not always coincide, with the Hecke pair $(SL_2(\mathbb{Q}_q), SL_2(\mathbb{Z}_q))$ providing one such example.}

\section{Introduction}

A \emph{Hecke pair} $(G, \Gamma)$ consists of a group $G$ and a subgroup $\Gamma \subseteq G$, called a \emph{Hecke subgroup}, for which every double coset $\Gamma g \Gamma$ is the union of finitely many left cosets.  Examples of Hecke subgroups include finite subgroups, finite-index subgroups and normal subgroups. It is many times insightful to think of Hecke subgroups as subgroups which are ``almost normal''. The \emph{Hecke algebra} $\hpalma(G, \Gamma)$ of  a Hecke pair $(G,\Gamma)$ is a $^*$-algebra of complex-valued functions over the set of double cosets $\Gamma \backslash G / \Gamma$, with suitable convolution product and involution. It generalizes the notion of the group algebra $\mathbb{C}(G / \Gamma)$ of the quotient group when $\Gamma$ is a normal subgroup.

For operator algebraists the interest in the subject of Hecke algebras was largely raised by the work of Bost and Connes \cite{bost connes palma} on phase transitions in number theory and their work has led several authors to study $C^*$-algebras which arise as completions of Hecke algebras. There are several canonical $C^*$-completions of a Hecke algebra $\hpalma(G, \Gamma)$ which one can consider: $C^*(G, \Gamma)$, $C^*(L^1(G, \Gamma))$, $pC^*(\overline{G})p$ and $C^*_r(G, \Gamma)$ (see \cite{tzanev palma} and \cite{schl palma}), and the question of when does $C^*(G, \Gamma)$ exist and when do some of these completions coincide has been studied by several authors (\cite{bost connes palma}, \cite{hall palma}, \cite{tzanev palma}, \cite{schl palma}, \cite{palma}, to name a few).

 An important question raised by Hall \cite{hall palma} where $C^*$-completions of Hecke algebras came to play an important role was if for a Hecke pair $(G, \Gamma)$ there is a correspondence between unitary representations of $G$ generated by the $\Gamma$-fixed vectors and nondegenerate $^*$-representations of $\hpalma(G, \Gamma)$, analogous to the known correspondence between representations of a group and of its group algebra. Whenever such a correspondence holds we say that $(G, \Gamma)$ satisfies \emph{Hall's equivalence}. It is known that Hall's equivalence does not hold in general \cite{hall palma}, and in fact a theorem of Kaliszewski, Landstad and Quigg \cite{schl palma} shows that Hall's equivalence holds precisely when $C^*(G, \Gamma)$ exists and $C^*(G, \Gamma) \cong C^*(L^1(G, \Gamma)) \cong pC^*(\overline{G})p$, which has been shown to be the case for several classes of Hecke pairs.

The primary goal of this article is to give a sufficient condition for the isomorphism $C^*(L^1(G, \Gamma)) \cong pC^*(\overline{G})p$ to hold and to combine this result with the results of \cite{palma} in order to establish Hall's equivalence for  several classes of Hecke pairs, including all Hecke pairs $(G, \Gamma)$ where $G$ is a nilpotent group. We will also show that the two $C^*$-completions $C^*(L^1(G, \Gamma))$ and $pC^*(\overline{G})p$ are in general different, with $(SL_2(\mathbb{Q}_q), SL_2(\mathbb{Z}_q))$ providing an example for which $C^*(L^1(G, \Gamma)) \ncong pC^*(\overline{G})p$.

The problem of deciding for which Hecke pairs the two completions $C^*(L^1(G, \Gamma))$ and $pC^*(\overline{G})p$ coincide is only partially understood. Several properties of the pair $(G, \Gamma)$ are known to force these two completions to coincide, and in this regard we recall a result by Kaliszewski, Landstad and Quigg \cite{schl palma} which states that $C^*(L^1(G, \Gamma)) \cong pC^*(\overline{G})p$ whenever the Schlichting completion $\overline{G}$ is a Hermitian group. We will generalize their result in Section \ref{quasi symmetric group algebras section palma} in a way that covers also all Hecke pairs for which $G$ or $\overline{G}$ has subexponential growth. For that we introduce the notion of a \emph{quasi-symmetric} group algebra: a locally compact group $G$ will be said to have a quasi-symmetric group algebra if for any $f \in C_c(G)$ the spectrum of $f^**f$ relative to $L^1(G)$ is in $\mathbb{R}^+_0$. It follows directly from the definition that Hermitian groups have a quasi-symmetric group algebra and it is a consequence of the work of Hulanicki (\cite{hul2 palma}, \cite{hul palma}) that this is also the case for groups of subexponential growth. We show that $C^*(L^1(G, \Gamma)) \cong pC^*(\overline{G})p$ whenever the Schlichting completion $\overline{G}$ has a quasi-symmetric group algebra.

Besides strictly generalizing Kaliszewski, Landstad and Quigg's result, as there are groups of subexponential growth which are not Hermitian, our result is easier to apply in practice since we can many times use it without any knowledge about the Schlichting completion $\overline{G}$, which is often hard to compute. In fact we will show that if $G$ has subexponential growth then so does $\overline{G}$, which means that knowledge about the original group $G$ is sufficient for applying our result. The relation between Hermitianess and subexponential growth will be discussed in Section \ref{further remarks quasi symmetry section palma}.

By combining our result on quasi-symmetric group algebras with the results of \cite{palma} and also a theorem of Tzanev \cite{tzanev palma}, we are able to establish in Section \ref{Halls equivalence section palma} that $C^*(G, \Gamma)$ exists and $C^*(G, \Gamma) \cong C^*(L^1(G, \Gamma)) \cong pC^*(\overline{G})p \cong C^*_r(G, \Gamma)$ for several classes of Hecke pairs, including all Hecke pairs $(G, \Gamma)$ where $G$ is a nilpotent group. Consequently, it follows that Hall's equivalence holds for all such classes of Hecke pairs.

It is natural to ask if there are examples of Hecke pairs for which we have $C^*(L^1(G, \Gamma)) \ncong  pC^*(\overline{G})p$.  According to \cite{schl palma}, Tzanev claims in private communication with Kaliszewski, Landstad and Quigg that the Hecke pair $(PSL_3(\mathbb{Q}_q), PSL_3(\mathbb{Z}_q))$ is such that $C^*(L^1(G, \Gamma)) \ncong  pC^*(\overline{G})p$, but no proof has been published and no other example seems to be known, as far as we know. We prove in Section \ref{counter-example section palma} that $C^*(L^1(G, \Gamma)) \ncong  pC^*(\overline{G})p$ for the Hecke pair $(PSL_2(\mathbb{Q}_q), PSL_2(\mathbb{Z}_q))$, as suggested by Kaliszewski, Landstad and Quigg in \cite{schl palma}, but following a different approach than the one they suggest which  does not use the representation theory of $PSL_2(\mathbb{Q}_q)$.

The author is thankful to his adviser Nadia Larsen for the very helpful discussions, suggestions and comments during the elaboration of this work.

\section{Preliminaries}
\label{preliminaries section palma}

\subsection{Hecke pairs and Hecke algebras}

We will mostly follow \cite{krieg palma} and \cite{schl palma} in what regards Hecke pairs and Hecke algebras and refer to these references for more details.

\begin{definition}
 Let $G$ be a group and $\Gamma$ a subgroup. The pair $(G , \Gamma)$ is called a \emph{Hecke pair} if every double coset $\Gamma g\Gamma$ is the union of finitely many right (and left) cosets. In this case, $\Gamma$ will be called a \emph{Hecke subgroup} of $G$.
\end{definition}

Given a Hecke pair $(G, \Gamma)$ we will denote by $L$ and $R$, respectively, the left and right coset counting functions, i.e.
\begin{align*}
 L(g):= |\Gamma g \Gamma / \Gamma|   < \infty \qquad\quad \text{and} \qquad\quad R(g) :=|\Gamma \backslash \Gamma g \Gamma|  < \infty\,.
\end{align*}
We recall that $L$ and $R$ are $\Gamma$-biinvariant functions which satisfy $L(g) = R(g^{-1})$ for all $g \in G$. Moreover, the function $\Delta: G \to \mathbb{Q^+}$ given by
\begin{equation*}
\label{def modular function Hecke algebra palma}
 \Delta(g) := \frac{L(g)}{R(g)}\,,
\end{equation*}
is a group homomorphism, usually called the \emph{modular function} of $(G, \Gamma)$.

\begin{definition}
 The \emph{Hecke algebra} $\hpalma(G, \Gamma)$ is the $^*$-algebra of finitely supported $\mathbb{C}$-valued functions on the double coset space $\Gamma \backslash G / \Gamma$ with the product and involution defined by
\begin{eqnarray*}
 (f_1*f_2)(\Gamma g \Gamma) & &:= \sum_{h\Gamma \in G / \Gamma} f_1(\Gamma h \Gamma)f_2(\Gamma h^{-1}g\Gamma)\,,\\
f^*(\Gamma g\Gamma) & &:= \Delta(g^{-1}) \overline{f(\Gamma g^{-1} \Gamma)}\,.\\
\end{eqnarray*}
\end{definition}

\begin{remark}
Some authors, including Krieg \cite{krieg palma}, do not include the factor $\Delta$ in the involution. Here we adopt the convention of Kaliszewski, Landstad and Quigg \cite{schl palma} in doing so, as it gives rise to a more natural $L^1$-norm. We note, nevertheless, that there is no loss (or gain) in doing so, because these two different involutions give rise to $^*$-isomorphic Hecke algebras.
\end{remark}

Given a Hecke pair $(G, \Gamma)$, the subgroup $R^\Gamma := \bigcap_{g \in G} g \Gamma g^{-1}$ is a normal subgroup of $G$ contained in $\Gamma$. A Hecke pair $(G, \Gamma)$ is called \emph{reduced} if $R^\Gamma = \{e\}$. As it is known, the pair $(G_r, \Gamma_r):= (G / R^\Gamma, \Gamma / R^\Gamma)$ is a reduced Hecke pair and the Hecke algebras $\hpalma(G, \Gamma) \cong \hpalma(G_r, \Gamma_r)$ are canonically isomorphic. For this reason the pair $(G_r, \Gamma_r)$ is called the \emph{reduction} of $(G, \Gamma)$, and the isomorphism of the corresponding Hecke algebras shows that it is enough to consider reduced Hecke pairs, a convention used by several authors. We will not use this convention however, since we aim at achieving general results based on properties of the original Hecke pair $(G, \Gamma)$, and not its reduction.

A natural example of a Hecke pair $(G, \Gamma)$ is given by a topological group $G$ and a compact open subgroup $\Gamma$. It is known that this type of examples are, in some sense, the general case: there is a canonical construction which associates to a given reduced Hecke pair $(G, \Gamma)$ a new Hecke pair $(\overline{G},\overline{\Gamma})$ with the following properties:
\begin{enumerate}
 \item $\overline{G}$ is a totally disconnected locally compact group;
 \item $\overline{\Gamma}$ is a compact open subgroup;
 \item the pair $(\overline{G}, \overline{\Gamma})$ is reduced;
 \item There is a canonical embedding $\theta: G \to \overline{G}$ such that $\theta(G)$ is dense in $\overline{G}$ and $\theta(\Gamma)$ is dense in $\overline{\Gamma}$. Moreover, $\theta^{-1}(\overline{\Gamma}) = \Gamma$;
\end{enumerate}

The pair $(\overline{G}, \overline{\Gamma})$ satisfies a well-known uniqueness property and is called the \emph{Schlichting completion} of $(G, \Gamma)$. For the details of this construction the reader is referred to \cite{tzanev palma} and \cite{schl palma}  (see also \cite{glockner palma} for a slightly different approach). We shall make a quick review of some known facts and we refer to the previous references for all the details.

Henceforward we will not write explicitly the canonical homomorphism $\theta$, and we will instead see $G$ as a dense subgroup of $\overline{G}$, identified with the image $\theta(G)$. The Schlichting completion $(\overline{G}, \overline{\Gamma})$ of a reduced Hecke pair $(G, \Gamma)$ satisfies the following additional property:

\begin{itemize}
 \item[5.] there are canonical bijections $G / \Gamma \to \overline{G} / \overline{\Gamma}$ and $ \Gamma \backslash G / \Gamma \to \overline{\Gamma} \backslash \overline{G} / \overline{\Gamma}$ given respectively by $g \Gamma \to g \overline{\Gamma}$ and $\Gamma g \Gamma \to \overline{\Gamma} g \overline{\Gamma}$.
\end{itemize}

If a Hecke pair $(G, \Gamma)$ is not reduced, its \emph{Schlichting completion} $(\overline{G}, \overline{\Gamma})$ is defined as the completion $(\overline{G_r}, \overline{\Gamma_r})$ of its reduction. There is then a canonical map with dense image $G \to \overline{G}$ which factors through $G_r$, and this map is an embedding if and only if $(G, \Gamma)$ is reduced, i.e. $G \cong G_r$.

Following \cite{schl palma}, we consider the normalized Haar measure $\mu$ on $\overline{G}$ (so that $\mu(\overline{\Gamma}) = 1$) and define the Banach $^*$-algebra $L^1(\overline{G})$ with the usual convolution product and involution. We denote by $p$ the characteristic function of $\overline{\Gamma}$, i.e. $p:= \chi_{\overline{\Gamma}}$, which is a projection in $C_c(\overline{G}) \subseteq L^1(\overline{G})$. Recalling \cite{tzanev palma} or \cite{schl palma}, we always have canonical $^*$-isomorphisms:
\begin{equation}
\label{several isomorophisms algebraic hecke algebra palma}
 \hpalma(G, \Gamma) \cong \hpalma(G_r, \Gamma_r) \cong \hpalma(\overline{G}, \overline{\Gamma}) \cong pC_c(\overline{G})p\,.
\end{equation}
The modular function $\Delta$ of a reduced Hecke pair $(G, \Gamma)$, defined by (\ref{def modular function Hecke algebra palma}), is simply the modular function of the group $\overline{G}$ restricted to $G$.

\subsection{$L^1$- and $C^*$-completions}

There are several ways of defining a $L^1$-norm in a Hecke algebra. One approach is to simply take the $L^1$-norm from $L^1(\overline{G})$, since the isomorphism in (\ref{several isomorophisms algebraic hecke algebra palma}) enables us to see the Hecke algebra as a subalgebra of $L^1(\overline{G})$. The completion of $\hpalma(G, \Gamma)$ with respect to this $L^1$-norm is isomorphic to the corner $pL^1(\overline{G})p$. Alternatively, one may take the following definition:

\begin{definition}
 The \emph{$L^1$-norm} on $\hpalma(G, \Gamma)$, denoted $\| \cdot \|_{L^1}$, is given by
\begin{equation*}
 \| f \|_{L^1} := \sum_{\Gamma g \Gamma \in \Gamma \backslash G / \Gamma} |f(\Gamma g \Gamma)|\, L(g)\,.
\end{equation*}
We will denote by $L^1(G, \Gamma)$ the completion of $\hpalma(G, \Gamma)$ under this norm.
\end{definition}

As observed in \cite{tzanev palma} or \cite{schl palma}, the two $L^1$-norms described above are the same. In fact we have canonical $^*$-isomorphisms
\begin{equation*}
\label{several isomorphisms L1 Hecke algebra palma}
 L^1(G, \Gamma) \cong L^1(\overline{G}, \overline{\Gamma}) \cong pL^1(\overline{G})p\,.
\end{equation*}

There are several canonical $C^*$-completions of $\hpalma(G, \Gamma)$. These are:

\begin{itemize}
 \item $C^*_r(G, \Gamma)$ - Called the \emph{reduced Hecke $C^*$-algebra}, it is the completion of $\hpalma(G, \Gamma)$ under the $C^*$-norm arising from a left regular representation (see \cite{tzanev palma}).
 \item $pC^*(\overline{G})p$ - The corner of the full group $C^*$-algebra $C^*(\overline{G})$. 
\item $C^*(L^1(G, \Gamma))$ - The enveloping $C^*$-algebra of $L^1(G, \Gamma)$.
\item $C^*(G, \Gamma)$ - The enveloping $C^*$-algebra (if it exists!) of $\hpalma(G, \Gamma)$. When it exists, it is usually called the \emph{full Hecke $C^*$-algebra}.
\end{itemize}

The various $C^*$-completions of $\hpalma(G, \Gamma)$ are related in the following way, through canonical surjective maps:
\begin{equation*}
 C^*(G,\Gamma)  \dashrightarrow C^*(L^1(G, \Gamma)) \longrightarrow pC^*(\overline{G})p \longrightarrow C^*_r(G, \Gamma)\,.
\end{equation*}

As was pointed out by Hall in \cite[Proposition 2.21]{hall palma}, the full Hecke $C^*$-algebra $C^*(G, \Gamma)$ does not have to exist in general. Nevertheless, its existence has been established for several classes of Hecke pairs (see, for example, \cite{schl palma}, \cite{hall palma} or \cite{palma}).

The question of whether some of these completions are actually the same has also been explored in the literature (\cite{bost connes palma}, \cite{schl palma}, \cite{tzanev palma}, \cite{palma}). We review here some of the main results.

The question of when one has the isomorphism $pC^*(\overline{G})p \cong C^*_r(G, \Gamma)$ was clarified by Tzanev, in \cite[Proposition 5.1]{tzanev palma}, to be a matter of amenability. As pointed out in \cite{schl palma}, there was a mistake in Tzanev's article (where it is assumed without proof that $C^*(L^1(G, \Gamma)) \cong pC^*(\overline{G})p$ is always true) which carries over to the cited Proposition 5.1. Nevertheless, Tzanev's proof holds if one just replaces $C^*(L^1(G, \Gamma))$ with $pC^*(\overline{G})p$, so that the correct statement of (a part of) his result becomes:

\begin{theorem}[Tzanev] 
\label{Tzanevs theorem palma}
$pC^*(\overline{G})p \cong C_r^*(G, \Gamma)$ if and only if $\overline{G}$ is amenable.
\end{theorem}

A known result concerning the isomorphism $C^*(L^1(G, \Gamma)) \cong pC^*(\overline{G})p$ was obtained by Kaliszewski, Landstad and Quigg in \cite[Theorem 5.14]{schl palma}, where they showed that this isomorphism holds when $\overline{G}$ is a Hermitian group.

In \cite{palma} we established the existence of $C^*(G, \Gamma)$ and also the isomorphism $C^*(G, \Gamma) \cong C^*(L^1(G, \Gamma))$ for several classes of Hecke pairs, recovering also various results in the literature in a unified approach.

An important result of Kaliszewski, Landstad and Quigg regarding the existence of $C^*(G, \Gamma)$ and the simultaneous isomorphisms $C^*(G, \Gamma) \cong C^*(L^1(G, \Gamma)) \cong pC^*(\overline{G})p$ will be discussed in the next subsection.

\subsection{Representation theory}

As it is well-known, for any group $G$ there is a canonical bijective correspondence (i.e. category equivalence) between unitary representations of $G$ and nondegenerate $^*$-representations of the group algebra $\mathbb{C}(G)$. Hall \cite{hall palma} asked whether something analogous was true for Hecke pairs, and the following definition is necessary in order to understand Hall's question:

\begin{definition}
 Let $G$ be a group and $\Gamma \subseteq G$ a subgroup. A unitary representation $\pi:G \to U(\mathscr{H})$ is said to be \emph{generated by its $\Gamma$-fixed vectors} if $\overline{\pi(G) \mathscr{H}^{\Gamma}} = \mathscr{H}$, where $\mathscr{H}^{\Gamma} = \{ \xi \in \mathscr{H}: \pi(\gamma) \xi = \xi\,,\;\;\text{for all}\;\;\gamma \in \Gamma\}$.
\end{definition}

The question which Hall posed in \cite{hall palma} is the following:

\begin{question}[Hall's equivalence]
 Let $(G, \Gamma)$ be a Hecke pair. Is there a category equivalence between nondegenerate $^*$-representations of $\hpalma(G, \Gamma)$ and unitary representations of $G$ generated by the $\Gamma$-fixed vectors?
\end{question}

Whenever there is an affirmative answer to this question, we shall say the Hecke pair $(G, \Gamma)$ satisfies \emph{Hall's equivalence}. In the work of Hall \cite{hall palma} and the subsequent work of Gl\" ockner and Willis \cite{glockner palma}, Hall's equivalence was studied and proven to hold under a certain form of positivity for some $^*$-algebraic bimodules. A more complete approach was further developed by Kaliszewski, Landstad and Quigg in \cite{schl palma}, where Hall's equivalence, positivity for certain $^*$-algebraic bimodules, and $C^*$-completions of Hecke algebras were all shown to be related. We briefly describe here the approach and results of \cite{schl palma} and the reader is referred to this reference for more details.

Let $(\overline{G}, \overline{\Gamma})$ be the Schlichting completion of a Hecke pair $(G, \Gamma)$. Following \cite[Section 5]{schl palma}, we have an inclusion of two imprimitivity bimodules (in Fell's sense):
\begin{equation*}
 _{C_c(\overline{G})pC_c(\overline{G})} \big(C_c(\overline{G})p \big)_{\hpalma(\overline{G}, \overline{\Gamma})} \qquad \subseteq \qquad _{L^1(\overline{G})pL^1(\overline{G})} \big(L^1(\overline{G})p \big)_{L^1(\overline{G}, \overline{\Gamma})}\,,
\end{equation*}
where the left and right inner products, $\langle \rangle_L$ and $\langle \rangle_R$, on these bimodules are given by multiplication within $L^1(\overline{G})$ by
\begin{equation*}
 \langle f\,,\,g\rangle_L = f*g^*\,,\qquad \qquad \langle f\,,\,g\rangle_R = f^**g\,.
\end{equation*}
A $^*$-representation $\pi$ of $\hpalma(G, \Gamma)$ is said to be $\langle \rangle_R$-\emph{positive} if
\begin{equation}
\label{positivity of rep palma}
 \pi(\langle f\,,\,f \rangle_R) \geq 0\,, \qquad\text{for all}\;\;f \in C_c(\overline{G})p\,.
\end{equation}
Similarly, a $^*$-representation $\pi$ of $L^1(G, \Gamma)$ is said to be $\langle \rangle_R$-\emph{positive} when condition (\ref{positivity of rep palma}) holds for all $f \in L^1(\overline{G})p$.

In \cite[Corollary 5.19]{schl palma} Kaliszewski, Landstad and Quigg proved that, for a reduced pair $(G, \Gamma)$, there exists a category equivalence between unitary representations of $G$ generated by the $\Gamma$-fixed vectors and the $\langle \rangle_R$-positive representations of $\hpalma(G, \Gamma)$. This is in fact true for non-reduced Hecke pairs $(G, \Gamma)$ as well, as follows from the following observation:

\begin{proposition}
 Let $(G, \Gamma)$ be a Hecke pair and $(G_r, \Gamma_r)$ its reduction. There exists a  category equivalence between unitary representations of $G$ generated by the $\Gamma$-fixed vectors and unitary representations of $G_r$ generated by the $\Gamma_r$-fixed vectors.

 The correspondence is as follows: a representation $\pi:G_r \to U(\mathscr{H})$ is mapped to the representation $\pi\circ q$, where $q: G \to G_r$ is the quotient map. Its inverse map takes a representation $\rho:G \to U(\mathscr{H})$ to the representation $\widetilde{\rho}$ of $G_r$ on the same Hilbert space, given by $\widetilde{\rho}([g]) := \rho(g)$.
\end{proposition}

\begin{proof}
First we observe that the assignment $\pi \mapsto \pi \circ q$ does indeed produce a unitary representation of $G$ generated by the $\Gamma$-fixed vectors. This is obvious since the spaces of fixed vectors $\mathscr{H}^{\Gamma_r}$ and $\mathscr{H}^{\Gamma}$ are the same.

Secondly, for the inverse assignment, we need to check that $\widetilde{\rho}$ is well-defined, which amounts to show that $\rho(g) = \rho(gh)$ for any $g \in G$ and $h \in R^{\Gamma}$. For any $s \in G$ and $ \xi \in \mathscr{H}^{\Gamma}$ we have
\begin{eqnarray*}
 \rho(gh) \rho(s) \xi & = & \rho(g) \rho(s) \rho(s^{-1}h s) \xi\\
& = & \rho(g) \rho(s) \xi\,,
\end{eqnarray*}
because $s^{-1} h s \in R^{\Gamma} \subseteq \Gamma$. Hence, $\rho(gh) = \rho(g)$ on the space $\overline{\pi(G)\mathscr{H}^{\Gamma}}$. Since $\rho$ is assumed to be generated by the $\Gamma$-fixed vectors, it follows that $\rho(gh) = \rho(g)$.

It is also easy to see that $\widetilde{\rho}$ is generated by the $\Gamma_r$-fixed vectors and it is clear from the definitions that these assignments are inverse of one another.

This correspondence does not change the Hilbert spaces of the representations, so that the intertwiners of representations are preserved in a canonical way. It can then be easily seen that this defines a category equivalence. \qed
\end{proof}

In the light of Kaliszewski, Landstad and Quigg's result, for a Hecke pair $(G, \Gamma)$ for which all $^*$-representations of $\hpalma(G, \Gamma)$ are $\langle \rangle_R$-positive, there exists a category equivalence between unitary representations of $G$ generated by the $\Gamma$-fixed vectors and nondegenerate $^*$-representations of $\hpalma(G, \Gamma)$. In other words, Hall's equivalence holds when all $^*$-representations of $\hpalma(G, \Gamma)$ are $\langle \rangle_R$-positive. Furthermore, Kaliszewski, Landstad and Quigg proved also the following relation between $\langle \rangle_R$-positivity and $C^*$-completions of Hecke algebras:

\begin{theorem}[\cite{schl palma} Corollary 5.11]
\label{Kal, Land, Quigg thm on Halls equivalence and C completions palma}
 Let $(G, \Gamma)$ be a Hecke pair.
\begin{enumerate}
 \item Every  $^*$-representation of $\hpalma(G, \Gamma)$ is $\langle \rangle_R$-positive if and only if  $C^*(G, \Gamma)$ exists and $C^*(G, \Gamma) \cong C^*(L^1(G, \Gamma)) \cong pC^*(\overline{G})p$.
\item Every $^*$-representation of $L^1(G, \Gamma)$ is $\langle \rangle_R$-positive if and only if $C^*(L^1(G, \Gamma)) \cong pC^*(\overline{G})p$.
\end{enumerate}
\end{theorem}

\subsection{Groups of subexponential growth}

 Let $G$ be a locally compact group with a Haar measure $\mu$. For a compact neighbourhood $V$ of $e$, the limit superior
\begin{align}
\label{growth rate limsup palma}
 \limsup_{n \to \infty} \mu(V^n)^{\frac{1}{n}}
\end{align}
will be called the \emph{growth rate} of $V$. Since $0 < \mu(V) \leq \mu(V^n)$ for all $n \in \mathbb{N}$ it is clear that the growth rate of $V$ is always greater or equal to one.

\begin{definition}
 A locally compact group $G$ is said to be of \emph{subexponential growth} if $\limsup_{n \to \infty} \mu(V^n)^{\frac{1}{n}} = 1$ for all compact neighbourhoods $V$ of $e$. Otherwise it is said to be of \emph{exponential growth}.
\end{definition}

The class of groups with subexponential growth is closed under taking closed subgroups \cite[Th\' eor\` eme I.2]{guiv palma} and quotients \cite[Th\' eor\` eme I.3]{guiv palma}. We observe that even though in \cite{guiv palma} the author is only working with compactly generated groups, the proofs of these results are general and hold for any locally compact group.

 It is known that if $G$ has subexponential growth as a discrete group, then it has subexponential growth with respect to any other locally compact topology \cite[Theorem 3.1]{hul palma}. The following is a slight generalization of this result, and the proof is done along similar lines:

\begin{proposition}
\label{subexp growth passes from a dense subgroup palma}
 Let $H$ be a dense subgroup of a locally compact group $\overline{H}$. If $H$ has subexponential growth as a discrete group, then $\overline{H}$ has subexponential growth in its locally compact topology.
\end{proposition}

\begin{proof}
Let $A \subseteq \overline{H}$ be a compact neighbourhood of $e$. First we claim that $H A = \overline{H}$. Since $A$ is a neighbourhood of $\{e\}$, there is an open set $U \subseteq A$ such that $e \in U$. To show that $H A = \overline{H}$, let $g \in \overline{H}$. Since $H$ is dense in $\overline{H}$ and $g(U \cap U^{-1})$ is open, it follows that there exists $h \in H \cap g(U \cap U^{-1})$. Thus, there exists $s \in U \cap U^{-1}$ such that $h = g s$, or equivalently, $g = hs^{-1}$. Since $s^{-1} \in U \cap U^{-1}$ we then have $g \in hU$, and thus $g \in hA$. Hence $\overline{H} = H A$.

 From the previous observation it follows that $\{hA\}_{h \in H}$ is a covering of the compact set $AA$, and since $A$ has non-empty interior there must exist a finite set $F \subset H$ such that $AA \subseteq F A$. Hence, we have $A^n \subseteq F^{n-1}A$, for all $n \geq 2$. Without loss of generality we can assume that $F$ contains the identity element. Now using the fact that $H$ has subexponential growth we obtain
\begin{eqnarray*}
 \limsup_{n \to \infty} \mu\big(A^n \big)^{\frac{1}{n}} & \leq &  \limsup_{n \to \infty} \mu(F^{n-1}A)^{\frac{1}{n}} \;\; \leq \;\; \limsup_{n \to \infty} |F^{n-1}|^{\frac{1}{n}} \mu\big(A\big)^{\frac{1}{n}} \;\; = \;\; 1\,. \;\;\; \qed
\end{eqnarray*}
\end{proof}

\begin{corollary}
\label{subexp growth of G implies the same for schl comp palma}
 Let $(G, \Gamma)$ be a discrete Hecke pair. If $G$ (or $G_r$) has subexponential growth, then so does $\overline{G}$.
\end{corollary}

\begin{proof}
 If $G$ has subexponential growth than so does any of its quotients, so in particular $G_r$ also has subexponential growth. If $G_r$ has subexponential growth then so does $\overline{G}$ by Proposition \ref{subexp growth passes from a dense subgroup palma}. \qed
 \end{proof}

Groups with subexponential growth are always unimodular \cite[Proposition 12.5.8]{palmer palma} and amenable  \cite[Section 12.6.18]{palmer palma}. 

The class of groups with subexponential growth includes all locally nilpotent groups and all $FC^-$-groups  \cite[Theorem 12.5.17]{palmer palma}. In particular, all abelian and all compact groups have subexponential growth.

\section{Quasi-symmetric group algebras}
\label{quasi symmetric group algebras section palma}

Given a $^*$-algebra $A$ and an element $a \in A$ we will use throughout this chapter the notations $\sigma_A(a)$ to denote the spectrum of $a$ relative to $A$, and $R_A(a)$ to denote the spectral radius of $a$ relative to $A$.

 Recall, for example from \cite{palmer palma}, that a $^*$-algebra $A$ is said to be:
\begin{itemize}
 \item \emph{Hermitian} if $\sigma_{A}(a) \subseteq \mathbb{R}$, for any self-adjoint element $a = a^*$ of $A$.
 \item \emph{symmetric} if $\sigma_{A}(a^*a) \subseteq \mathbb{R}^+_0$, for any $a \in A$.
\end{itemize}

It is an easy fact that symmetry implies Hermitianess. The two properties are equivalent for Banach $^*$-algebras, as asserted by the Shiralli-Ford theorem \cite{shirali ford palma}.

Recall also that a locally compact group $G$ is called \emph{Hermitian} if $L^1(G)$ is a Hermitian (equivalently, symmetric) Banach $^*$-algebra. The class of Hermitian groups satisfies some known closure properties, some of which we list below:

\begin{enumerate}
 \item The class of Hermitian groups is closed under taking open subgroups and quotients \cite[Theorem 12.5.18]{palmer palma}.
 \item Let $1 \to H \to G \to G/ H \to 1$ be an extension of locally compact groups. If $H$ is Hermitian and $G/ H$ is finite, then $G$ is Hermitian \cite[Theorem 12.5.18]{palmer palma}.
\end{enumerate}

The class of groups we are interested in this chapter arise by relaxing the condition of symmetry on the group algebra:

\begin{definition}
 Let $G$ be a locally compact group. We will say that the group algebra $L^1(G)$ is \emph{quasi-symmetric} if $\sigma_{L^1(G)}(f^**f) \subseteq \mathbb{R}^+_0$ for any compactly supported continuous function $f$.
\end{definition}

Clearly, Hermitian groups have a quasi-symmetric group algebra. Another important class of groups with this property is that of groups with subexponential growth, which comes as a consequence of the work of Hulanicki (for discrete groups this was established in \cite{hul2 palma}):

\begin{proposition}
\label{subexp growth implies quasi-symmetry palma}
 If $G$ is a locally compact group with subexponential growth, then $L^1(G)$ is quasi-symmetric.
\end{proposition}

\begin{proof}
Let $\lambda:L^1(G) \to B(L^2(G))$ denote the left regular representation of $L^1(G)$. Hulanicki proved in \cite{hul palma} that if $G$ has subexponential growth then
\begin{equation}
\label{hulainickis spectral radius palma}
 R_{L^1(G)}(f) = \|\lambda(f)\|\,,
\end{equation}
for any self-adjoint $f= f^*$ continuous function of compact support. Moreover, Barnes shows in \cite{barnes palma} (a  result which he credits to Hulanicki \cite{hul3 palma}) that if $A$ is a Banach $^*$-algebra, $B \subseteq A$ a $^*$-subalgebra and if $\pi:A \to B(\mathscr{H})$ is a faithful $^*$-representation such that
\begin{equation*}
 R_{A}(b) = \|\pi(b)\|\,,
\end{equation*}
for all self-adjoint elements $b = b^*$ in $B$, then $\sigma_{A}(b) = \sigma_{B(\mathscr{H})}(\pi(b))$ for every $b \in B$.

Considering $A$ and $B$ to be $L^1(G)$ and $C_c(G)$ respectively, we see from (\ref{hulainickis spectral radius palma}) that by taking $\pi$ to be $\lambda$ we immediately get that $\sigma_{L^1(G)}(f^**f) = \sigma_{B(L^2(G))}(\lambda(f^* * f))$ for any $f \in C_c(G)$. Thus, $\sigma_{L^1(G)}(f^**f) \subseteq \mathbb{R}_0^+$ for $f \in C_c(G)$, i.e. $L^1(G)$ is quasi-symmetric. \qed
 \end{proof}

The following result is the main result in this section and explains the reason for considering quasi-symmetric group algebras in the context of $C^*$-completions of Hecke pairs.

\begin{theorem}
\label{C completions for the quasi symmetric grp alg palma}
 Let $(G, \Gamma)$ be a Hecke pair. If $\overline{G}$ has a quasi-symmetric group algebra, then
\begin{equation*}
 C^*(L^1(G, \Gamma)) \cong pC^*(\overline{G})p\,.
\end{equation*}
In particular, there is a category equivalence between $^*$-representations of $L^1(G, \Gamma)$ and unitary representations of $G$ generated by the $\Gamma$-fixed vectors.
\end{theorem}

\begin{lemma}
\label{inclusion of spectrum palma}
 Let $(G, \Gamma)$ be a Hecke pair and $f \in pL^1(\overline{G})p$. We have that $\sigma_{pL^1(\overline{G})p}(f) \subseteq \sigma_{L^1(\overline{G})}(f)$.
\end{lemma}

\begin{proof}
Let us denote by $L^1(\overline{G})^{\dag}$ the minimal unitization of $L^1(\overline{G})$ and let ${\bf 1} \in L^1(\overline{G})^{\dag}$ be its unit. Let $\lambda \in \mathbb{C}$ and suppose that $f - \lambda {\bf 1}$ is invertible in $L^1(\overline{G})^{\dag}$. We want to prove that $f - \lambda p$ is invertible in $pL^1(\overline{G})p$. Invertibility of $f - \lambda {\bf 1}$ in $L^1(\overline{G})^{\dag}$ means that there exist $g \in L^1(\overline{G})$ and $\beta \in \mathbb{C}$ such that $ {\bf 1} = (f - \lambda {\bf 1})(g + \beta {\bf 1})$. Hence we have
\begin{eqnarray*}
 p & = & p(f - \lambda {\bf 1})(g + \beta {\bf 1})p \;\; = \;\; (pf - \lambda p)(gp +  \beta p)\\
& = & (fp - \lambda p)(gp +  \beta p) \;\; = \;\; (f - \lambda p)p(gp +  \beta p)\\
& = &  (f - \lambda p)(pgp +  \beta p)\,.
\end{eqnarray*}
Hence, $f - \lambda p$ is invertible in $pL^1(\overline{G})p$ and this finishes the proof. \qed
\end{proof}

\begin{proof}[Theorem \ref{C completions for the quasi symmetric grp alg palma}]
Due to the canonical isomorphism $L^1(G, \Gamma) \cong  pL^1(\overline{G}) p$, it is enough to prove that $C^*(pL^1(\overline{G}) p) \cong  pC^*(\overline{G}) p $. By \cite[Corollary 5.11]{schl palma} we only need to show that every representation of $pL^1(\overline{G})p$ is $\langle \rangle_R$-positive. Let $\pi:pL^1(\overline{G})p \to B(\mathscr{H})$ be a $^*$-representation and $f \in L^1(\overline{G})p$. Let $\{g_n\}_{n \in \mathbb{N}}$ be a sequence of functions in $C_c(\overline{G})p$ such that $g_n \to f$ in $L^1(\overline{G})$. Then, we also have $g_n^**g_n \to f^**f$ in $L^1(\overline{G})$. It is a standard fact that
\begin{equation*}
 \sigma_{B(\mathscr{H})}(\pi(g_n^**g_n)) \subseteq \sigma_{pL^1(\overline{G})p}(g^*_n*g_n)\,,
\end{equation*}
  and  by Lemma  \ref{inclusion of spectrum palma} we have $\sigma_{pL^1(\overline{G})p}(g^*_n* g_n) \subseteq \sigma_{L^1(\overline{G})}(g^*_n* g_n)$. Moreover, since $L^1(\overline{G})$ is quasi-symmetric we have that $\sigma_{L^1(\overline{G})}(g^*_n* g_n) \subseteq \mathbb{R}^+_0$. All these inclusions combined give
\begin{equation*}
  \sigma_{B(\mathscr{H})}(\pi(g_n^**g_n)) \subseteq \sigma_{pL^1(\overline{G})p}(g^*_n*g_n) \subseteq \sigma_{L^1(\overline{G})}(g^*_n* g_n) \subseteq \mathbb{R}^+_0\,,
\end{equation*}
and therefore $\pi(g^*_n* g_n)$ is a positive operator for every $n \in \mathbb{N}$. Thus, the limit $\pi(f^**f) = \lim \pi(g^*_n* g_n)$ is also a positive operator. In other words, $\pi(\langle f, f \rangle_R) \geq 0$. \qed
\end{proof}

As a consequence we immediately recover Kaliszewski, Landstad and Quigg's original result and also that $C^*(L^1(G, \Gamma)) \cong pC^*(\overline{G})p \cong C^*_r(G, \Gamma)$ for Hecke pairs arising from groups of subexponential growth.

\begin{corollary}[\cite{schl palma}, Theorem 5.14]
  Let $(G, \Gamma)$ be a Hecke pair. If $\overline{G}$ is Hermitian, then $C^*(L^1(G, \Gamma)) \cong pC^*(\overline{G})p$.
\end{corollary}

\begin{corollary}
\label{subexponential growth imply completions of Hecke algebra cor palma}
 Let $(G, \Gamma)$ be a Hecke pair. If one of the groups $G$, $G_r$ or $\overline{G}$ has subexponential growth, then $C^*(L^1(G, \Gamma)) \cong pC^*(\overline{G})p \cong C^*_r(G,\Gamma)$.
\end{corollary}

\begin{proof}
 By Corollary \ref{subexp growth of G implies the same for schl comp palma}, if $G$ or $G_r$ has subexponential growth, then so does $\overline{G}$ in its totally disconnected locally compact topology. Since $\overline{G}$ has subexponential growth, we have that $L^1(\overline{G})$ is quasi-symmetric and therefore $C^*(L^1(G, \Gamma)) \cong pC^*(\overline{G})p$ by Theorem \ref{C completions for the quasi symmetric grp alg palma}. The isomorphism $pC^*(\overline{G})p \cong C^*_r(G, \Gamma)$ follows from Tzanev's theorem (Theorem \ref{Tzanevs theorem palma} in the present work), due to the fact that subexponential growth implies the amenability of the group $\overline{G}$. \qed
\end{proof}

\section{Remark on subexponential growth for Hecke pairs}

The hypothesis in Corollary \ref{subexponential growth imply completions of Hecke algebra cor palma} require that one of the groups $G$, $G_r$ or $\overline{G}$ has subexponential growth. A natural question to ask is if there is a reasonable definition of subexponential growth for a Hecke pair $(G, \Gamma)$. Such a definition should heuristically mean that the ``quotient'' $G / \Gamma $ has subexponential growth, and could in principle be taken as the hypothesis in Corollary \ref{subexponential growth imply completions of Hecke algebra cor palma} and render a more general result. We say more general because one should expect that subexponential growth of $G$ (or $G_r$ or $\overline{G}$) would imply subexponential growth of the pair $(G, \Gamma)$, since this property passes to quotients.

  As we shall see, it is possible to give such a definition, but this turns out to be equivalent to the Schlichting completion $\overline{G}$ having subexponential growth, as it is intuitively expected: since $\overline{\Gamma}$ is compact, subexponential growth of $\overline{G} / \overline{\Gamma}$ is equivalent to subexponential growth of $\overline{G}$.

Let $(G, \Gamma)$ be a Hecke pair. Given a finite subset $A \subseteq \Gamma \backslash G /\Gamma$ of double cosets, we will denote by $L(A):= \sum_{[g] \in A} L(g)$ the total number of left cosets inside $A$. Also, if $A, B \subseteq \Gamma \backslash G / \Gamma$ are finite subsets we will denote by $AB  \subseteq \Gamma \backslash G / \Gamma$ the set
\begin{equation*}
 AB := \{  [g] \in \Gamma \backslash G / \Gamma: \Gamma g \Gamma \subseteq \Gamma a \Gamma b\Gamma,\; \text{for some}\;\; [a] \in A, [b] \in B \}\,,
\end{equation*}
which is itself a finite set. Moreover, for $n \in \mathbb{N}$ we define $A^n$ inductively as $A^n := AA^{n-1}$, with $A^0 := A$.

\begin{definition}
\label{subexp growth for Hecke pairs palma}
 We will say that a Hecke pair $(G, \Gamma)$ has \emph{subexponential growth} if for every finite set $A \subseteq \Gamma \backslash G / \Gamma$ we have
\begin{equation*}
 \limsup_{n \to \infty} L(A^n)^{\frac{1}{n}} = 1\,.
\end{equation*}
\end{definition}

We note that when $\Gamma$ is a normal subgroup Definition \ref{subexp growth for Hecke pairs palma} means precisely that the quotient group $G / \Gamma$ has subexponential growth.

\begin{proposition}
 The following statements are equivalent:
\begin{enumerate}
 \item[(i)] $(G, \Gamma)$ has subexponential growth.
 \item[(ii)] $(G_r, \Gamma_r)$ has subexponential growth.
 \item[(iii)] $(\overline{G}, \overline{\Gamma})$ has subexponential growth.
 \item[(iv)] $\overline{G}$ has subexponential growth.
\end{enumerate}

\end{proposition}

\begin{proof}
 It is clear that $(G, \Gamma)$, $(G_r, \Gamma_r)$ and $(\overline{G}, \overline{\Gamma})$ have exactly the same growth rate, since we can canonically identify the double coset spaces $\Gamma \backslash G / \Gamma$, $\Gamma_r \backslash G_r / \Gamma_r$ and $\overline{\Gamma} \backslash \overline{G} / \overline{\Gamma}$, and also the corresponding Hecke algebras $\hpalma(G, \Gamma)$, $\hpalma(G_r, \Gamma_r)$ and $\hpalma(\overline{G}, \overline{\Gamma})$. So it remains to see that $(iii) \Leftrightarrow (iv)$.

To see that $(iii) \Rightarrow (iv)$ let us consider a compact neighbourhood $A \subseteq \overline{G}$. Since the set $\overline{\Gamma} A \overline{\Gamma} \subseteq \overline{G}$ is both compact and open, it follows that $ B:= \overline{\Gamma} \backslash A / \overline{\Gamma}$ is a finite set of double cosets, and it is not difficult to see that $\overline{\Gamma} \backslash A^n / \overline{\Gamma} \subseteq B^n$.

With $\mu$ being the normalized Haar measure on $\overline{G}$ so that $\mu(\overline{\Gamma}) = 1$, we have that
\begin{eqnarray*}
 L(\overline{\Gamma} \backslash A^n / \overline{\Gamma}) & = & \sum_{ [g] \in \overline{\Gamma} \backslash A^n / \overline{\Gamma}} L(g) \;\; = \;\; \sum_{[g] \in \overline{\Gamma} \backslash A^n / \overline{\Gamma}} \mu(\overline{\Gamma} g \overline{\Gamma})\\
& = & \mu \big(\bigcup_{[g] \in \overline{\Gamma} \backslash A^n / \overline{\Gamma}} \overline{\Gamma} g \overline{\Gamma}\; \big) \;\; = \;\; \mu(\overline{\Gamma} A^n \overline{\Gamma})\,.
\end{eqnarray*}
Hence, from the fact that $A^n \subseteq \overline{\Gamma} A^n \overline{\Gamma}$ and the assumption that $(\overline{G}, \overline{\Gamma})$ has subexponential growth, it follows that
\begin{eqnarray*}
 \limsup_{n \to \infty} \mu(A^n)^{\frac{1}{n}} & \leq & \limsup_{n \to \infty} \mu(\overline{\Gamma}A^n \overline{\Gamma})^{\frac{1}{n}} \;\; = \;\; \limsup_{n \to \infty} L(\overline{\Gamma} \backslash A^n / \overline{\Gamma})^{\frac{1}{n}}\\
& \leq &  \limsup_{n \to \infty} L(B^n)^{\frac{1}{n}} \;\;=\;\; 1\,.
\end{eqnarray*}
Let us now prove the direction $(iv) \Rightarrow (iii)$. For any given set $A \subseteq \overline{\Gamma} \backslash \overline{G} / \overline{\Gamma}$ there is a correspondent set $\widetilde{A} \subseteq \overline{G}$, consisting of the union of all the double cosets in $A$, i.e. $\widetilde{A}:= \{g \in \overline{G}: [g] \in A \}$. It is not difficult to see that $\widetilde{AB} = \widetilde{A}\widetilde{B}$, for any $A, B \in \overline{\Gamma} \backslash \overline{G} / \overline{\Gamma}$, and therefore $\widetilde{A^n} = \big(\widetilde{A} \big)^n$.

 Let us take a finite set $A  \subseteq \overline{\Gamma} \backslash \overline{G} / \overline{\Gamma}$. We have
\begin{equation*}
 \limsup_{n \to \infty} L(A^n)^{\frac{1}{n}}  =   \limsup_{n \to \infty} \mu \big(\widetilde{A^n} \big)^{\frac{1}{n}}  =  \limsup_{n \to \infty} \mu \big(\big(\widetilde{A}\big)^n \big)^{\frac{1}{n}} = 1\,. \qquad \qed
\end{equation*}

\end{proof}

\section{Further remarks on groups with a quasi-symmetric group algebra}
\label{further remarks quasi symmetry section palma}

The classes of Hermitian groups and groups with subexponential growth are in general different. On one side, there are examples of Hermitian groups which do not have subexponential growth, such as the affine group of the real line $\mathrm{Aff}(\mathbb{R}):= \mathbb{R} \rtimes \mathbb{R}^*$, with its usual topology as a (connected) Lie group, as shown by Leptin \cite{leptin palma}. On the other side, there are examples of groups with subexponential growth which are not Hermitian, such as the Fountain-Ramsay-Williamson group \cite{fountain palma}, which is the discrete group with the presentation
\begin{equation*}
 \big\langle \{u_j\}_{j \in \mathbb{N}} \;|\; u_j^2 = e \;\;\text{and}\;\; u_iu_ju_ku_j = u_ju_ku_ju_i \;\;\forall i,j < k \in \mathbb{N} \big\rangle\,.
\end{equation*}
Fountain, Ramsay and Williamson showed that this group is not Hermitian despite being locally finite (thus, having subexponential growth). Another such example was given by Hulanicki in \cite{hul4 palma}.

Using these examples we can show that the class of groups with a quasi-symmetric group algebra is strictly larger than the union of the classes of Hermitian groups and groups with subexponential growth. In that regard we have the following result:

\begin{proposition}
\label{H times L example palma}
 Let $H$ be a Hermitian locally compact group with exponential growth and let $L$ be a discrete locally finite group which is not Hermitian. The locally compact group $G:= H \times L$ has a quasi-symmetric group algebra, but it is neither Hermitian nor has subexponential growth.

An example of such a group is given by taking $H := \mathrm{Aff}(\mathbb{R})$ and $L$ the  Fountain-Ramsay-Williamson group.
\end{proposition}

\begin{proof}
 Let us first prove that $G:=H \times L$ has a quasi-symmetric group algebra. Given a function $f \in C_c(G)$, the product $f^**f$ also has compact support, and since $L$ is discrete, the support of $f^**f$ must lie inside some set of the form $H \times F$, where $F \subseteq L$ is a finite set. Since $L$ is locally finite, $F$ generates a finite subgroup $\langle F \rangle \subseteq G$. Now $H \times \langle F \rangle$ is an open subgroup of $G$, so that
\begin{equation*}
 L^1(H \times \langle F \rangle ) \subseteq L^1(G)\,.
\end{equation*}
The group $H \times \langle F \rangle$ is Hermitian, being a finite extension of a Hermitian group, and therefore $\sigma_{L^1(H \times \langle F \rangle )}(f^**f) \subseteq \mathbb{R}_0^+$. This implies that
\begin{equation*}
\sigma_{L^1(G)}(f^**f) \subseteq \sigma_{L^1(H \times \langle F \rangle )}(f^**f) \subseteq \mathbb{R}_0^+\,,
\end{equation*}
which shows that $G$ is quasi-symmetric.

This group is not Hermitian, because it has a quotient ($L$) which is not Hermitian, and it does not have subexponential growth because it has a quotient ($H$) which does not have subexponential growth. \qed
\end{proof}

Since in the present work we are directly concerned with totally disconnected groups (because of the Schlichting completion), it would be interesting to know if there are examples of totally disconnected groups with a quasi-symmetric group algebra, but which are not Hermitian nor have subexponential growth. We do not know the answer to this question. The example considered in Proposition \ref{H times L example palma} is of course not totally disconnected since $\mathrm{Aff}(\mathbb{R})$ is a connected group. But in view of Proposition \ref{H times L example palma}, it would suffice to answer affirmatively the following more fundamental problem:

\begin{question}
\label{open question on Her tot disc exp growth palma}
 Is there any Hermitian, totally disconnected group, with exponential growth?
\end{question}

As we pointed out above, there are examples of locally compact groups (even connected ones) which are Hermitian and have exponential growth, such as $\mathrm{Aff}(\mathbb{R})$, but the question of whether this can happen in the totally disconnected setting is, as far as we understand, still open. In the discrete case, Palmer \cite{palmer palma} claims that all examples of discrete groups which are known to be Hermitian actually have subexponential growth (even more, polynomial growth).

An affirmative answer to question \ref{open question on Her tot disc exp growth palma} would make, as we pointed out, the class of groups with a quasi-symmetric group algebra richer than the union of the classes of Hermitian and subexponential growth groups.

 On the other side, a negative answer to the above question would mean that any Hermitian totally disconnected group necessarily has subexponential growth and is therefore amenable, and thus would bring new evidence for the long standing conjecture that all Hermitian groups are amenable (\cite{palmer palma}), which is known to be true in the connected case \cite[Theorem 12.5.18 (e)]{palmer palma}. In fact, a negative answer to \ref{open question on Her tot disc exp growth palma} in the discrete case alone would, through the theory of extensions, imply that all Hermitian groups with an open connected component are amenable.

 The fact that we do not know of any totally disconnected group with a quasi-symmetric group algebra which does not have subexponential growth is not a drawback in any way. In fact, the class of groups with subexponential growth is already very rich by itself and will be used to give meaningful examples in Hecke $C^*$-algebra theory and Hall's equivalence in the next section.

\section{Hall's equivalence}
\label{Halls equivalence section palma}

Combining the results of \cite{palma} on the existence of $C^*(G, \Gamma)$ and the isomorphism $C^*(G, \Gamma) \cong C^*(L^1(G, \Gamma))$, with the results on this paper on groups of subexponential growth and also Tzanev's theorem, we are able to establish that
\begin{align*}
 C^*(G, \Gamma) \cong C^*(L^1(G, \Gamma)) \cong pC^*(\overline{G})p \cong  C^*_r(G, \Gamma)\,,
\end{align*}
for several classes of Hecke pairs, including all Hecke pairs $(G, \Gamma)$ where $G$ is a nilpotent group. As consequence, Kaliszewski, Landstad and Quigg's theorem (Theorem \ref{Kal, Land, Quigg thm on Halls equivalence and C completions palma} in the present work) yields that Hall's equivalence is satisfied for all such classes of Hecke pairs.

\begin{proposition}
 If a group $G$ satisfies one of the following generalized nilpotency properties:
\begin{itemize}
 \item $G$ is finite-by-nilpotent, or
 \item $G$ is hypercentral, or
 \item all subgroups of $G$ are subnormal,
\end{itemize}
then for any Hecke subgroup $\Gamma \subseteq G$ we have that $C^*(G, \Gamma)$ exists and
\begin{equation*}
 C^*(G, \Gamma) \cong C^*(L^1(G, \Gamma)) \cong p C^*(\overline{G})p \cong C^*_r(G, \Gamma)\,.
\end{equation*}
In particular, Hall's equivalence holds with respect to any Hecke subgroup.
\end{proposition}

\begin{proof} As discussed in \cite[Classes 5.8, 5.9, 5.5]{palma} for every Hecke pair $(G, \Gamma)$ where $G$ satisfies one of the aforementioned properties we have that the full Hecke $C^*$-algebra exists and $C^*(G, \Gamma) \cong C^*(L^1(G, \Gamma))$.

We claim that if $G$ has one of the three properties above, it must have subexponential growth. If $G$ is finite-by-nilpotent, then by definition $G$ is a nilpotent extension of a finite group, and since nilpotent groups have subexponential growth, then so does $G$. If $G$ is hypercentral or all subgroups of $G$ are subnormal, then it is known that $G$ is locally nilpotent and therefore must have subexponential growth (see \cite{hul2 palma}). Consequently, by Corollary \ref{subexponential growth imply completions of Hecke algebra cor palma} we must have $C^*(L^1(G,\Gamma)) \cong pC^*(\overline{G})p \cong C^*_r(G, \Gamma)$. \qed
\end{proof}

If we restrict ourselves to finite subgroups $\Gamma \subseteq G$ we get a similar result for other classes of groups:

\begin{proposition}
 If a group $G$ satisfies one of the following properties:
\begin{itemize}
 \item $G$ is an $FC$-group, or
 \item $G$ is locally nilpotent, or
 \item $G$ is locally finite,
\end{itemize}
then for any finite subgroup $\Gamma \subseteq G$ we have that $C^*(G, \Gamma)$ exists and
\begin{equation*}
 C^*(G, \Gamma) \cong C^*(L^1(G, \Gamma)) \cong p C^*(\overline{G})p \cong C^*_r(G, \Gamma)\,.
\end{equation*}
In particular, Hall's equivalence holds with respect to any finite subgroup.
\end{proposition}

\begin{proof}
 As discussed in \cite[Classes 5.10, 5.11, 5.12]{palma} for every group $G$ that satisfies one of the aforementioned properties we have that, for any finite subgroup $\Gamma$, the full Hecke $C^*$-algebra exists and we have $C^*(G, \Gamma) \cong C^*(L^1(G, \Gamma))$. Also if $G$ has one of the three properties above, it must have subexponential growth (for $FC$- and locally nilpotent groups see \cite{hul2 palma}, and for locally finite groups it is obvious). Consequently, by Corollary \ref{subexponential growth imply completions of Hecke algebra cor palma} we must have $C^*(L^1(G,\Gamma)) \cong pC^*(\overline{G})p \cong C^*_r(G, \Gamma)$. \qed
\end{proof}

\begin{remark}
 The results above show that Hall's equivalence holds for any Hecke pair $(G, \Gamma)$ where $G$ satisfies a certain generalized nilpotency property. An analogous result for the class of solvable groups cannot hold. In \cite[Example 3.4]{tzanev palma} Tzanev gave an example of a Hecke pair $(G, \Gamma)$ where $G$ is solvable but for which $C^*(G, \Gamma)$ does not exist, and consequently Hall's equivalence does not hold. The example consists of the infinite dihedral group $G := \mathbb{Z} \rtimes (\mathbb{Z} / 2 \mathbb{Z})$ together with $\Gamma := \mathbb{Z} / 2\mathbb{Z}$.
\end{remark}

\section{A counter-example}
\label{counter-example section palma}

In the previous sections we have established a sufficient condition for the isomorphism $C^*(L^1(G, \Gamma)) \cong pC^*(\overline{G})p$ to hold, namely whenever $\overline{G}$ has a quasi-symmetric group algebra. A natural question to ask is the following: is it even possible that $C^*(L^1(G, \Gamma)) \ncong pC^*(\overline{G})p$ ?  We will now show that $C^*(L^1(G, \Gamma)) \ncong pC^*(\overline{G})p$ for the Hecke pair $(PSL_2(\mathbb{Q}_q), PSL_2(\mathbb{Z}_q))$,
where $q$ denotes a prime number and $\mathbb{Q}_q$, $\mathbb{Z}_q$ denote respectively the field of $q$-adic numbers and the ring of $q$-adic integers. It was already asked in \cite[Example 10.8]{schl palma} if $C^*(L^1(G, \Gamma)) \ncong pC^*(\overline{G})p$ for this Hecke pair and a strategy to achieve this result was designed. Our approach is nevertheless different from the approach suggested in \cite{schl palma} since we make no use of the representation theory of $PSL_2(\mathbb{Q}_q)$.

As we remarked in the introduction, Tzanev has claimed that the Hecke pair $(PSL_3(\mathbb{Q}_q), PSL_3(\mathbb{Z}_q))$ gives another example, but no proof has been published.

\begin{theorem}
 Let $q$ be a prime number and $\mathbb{Q}_q$ and $\mathbb{Z}_q$ denote respectively the field of $q$-adic numbers and the ring of $q$-adic integers. For the Hecke pair $(G, \Gamma):=(PSL_2(\mathbb{Q}_q), PSL_2(\mathbb{Z}_q))$ we have that $C^*(L^1(G, \Gamma)) \ncong pC^*(G)p$ .
\end{theorem}

\begin{proof}
For ease of reading and so that no confusion arises between the prime number $q$ and the projection $p$, we will throughout this proof denote the projection $p$ by $P$. Thus, our goal is to prove that $C^*(L^1(G, \Gamma)) \ncong PC^*(G)P$.

The pair $(PSL_2(\mathbb{Q}_q), PSL_2(\mathbb{Z}_q))$ coincides with its own Schlichting completion (see \cite{schl palma}) and is the reduction of the pair $(SL_2(\mathbb{Q}_q), SL_2(\mathbb{Z}_q))$. For ease of reading we will work with pair $(SL_2(\mathbb{Q}_q), SL_2(\mathbb{Z}_q))$ in this proof.

 The structure of the Hecke algebra $\hpalma(G, \Gamma)$ is well-known, and for convenience we will mostly refer to Hall \cite[Section 2.1.2.1]{hall palma} whenever we need to. Letting
\begin{equation*}
 x_n :=\begin{pmatrix}
  q^n & 0\\
  0 & q^{-n}
 \end{pmatrix}\,,
\end{equation*}
 it is known (\cite[Prop. 2.9]{hall palma}) that every double coset $\Gamma s \Gamma$ can be uniquely represented as $\Gamma x_n \Gamma$ for some $n \in \mathbb{N}$.

For each $0 \leq k \leq q-1$ let us denote by $y_k \in G$ the matrix
\begin{equation*}
 y_k :=\begin{pmatrix}
  q & k\\
  0 & q^{-1}
 \end{pmatrix}\,,
\end{equation*}
and let us take $g \in L^1(G)P$ as the element $g:= y_0P + y_1P + \dots + y_{q-1}P$, and $f := P + g$. We then have
\begin{eqnarray*}
 f^*f & = &  (P+g)^*(P+g) \;\; = \;\; P + g^*P + Pg +g^*  g\\
& = & P + \sum_{k = 0}^{q-1} Py_k^{-1}P + \sum_{k = 0}^{q-1} Py_kP + \sum_{i,j = 0}^{q-1} P y_i^{-1} y_jP\\
& = & (q+1)P +  \sum_{k = 0}^{q-1} Py_k^{-1}P + \sum_{k = 0}^{q-1} Py_kP + \sum_{\substack{i,j = 0\\ i \neq j}}^{q-1} P y_i^{-1} y_jP\,.
\end{eqnarray*}
As it is know (see for example \cite[Props. 2.10 and 2.12]{hall palma}), in $\hpalma(G, \Gamma)$ the modular function is trivial and each double coset is self-adjoint. Hence we can write
\begin{eqnarray*} 
f^*f & = & (q+1)P +  2\sum_{k = 0}^{q-1} Py_kP + 2\sum_{\substack{i,j = 0\\ i < j}}^{q-1} P y_i^{-1} y_jP\,.
\end{eqnarray*}
We now notice that, from \cite[Prop. 2.9]{hall palma}, we have $\Gamma y_k \Gamma =  \Gamma x_1 \Gamma$, and therefore $Py_kP = Px_1 P$. Moreover, for $0 \leq i < j \leq q-1$, we have that
\begin{equation*}
 y_i^{-1}y_j = \begin{pmatrix}
  1 & (j-i)q^{-1}\\
  0 & 1
 \end{pmatrix}\,,
\end{equation*}
and again from \cite[Prop. 2.9]{hall palma} we conclude that $Py_i^{-1}y_jP = Px_1P$. Hence, we get
\begin{eqnarray*}
 f^*f & = & (q+1)\,P +  2q\, Px_1P + 2\frac{(q-1)q}{2}\, P x_1P\\
 & = & (q+1) \,P + (q^2 + q)\,Px_1P\,.
\end{eqnarray*}
It is well known that $\hpalma(G, \Gamma)$ is commutative (see for example \cite[Section 2.2.3.2]{hall palma}) and all of its characters have been explicitly described. Following \cite[Example 10.8]{schl palma} the characters of $\hpalma(G, \Gamma)$ are precisely all the functions $\pi_z : \hpalma(G, \Gamma) \to \mathbb{C}$ such that
\begin{equation*}
 \pi_z(Px_mP) = \frac{1-qz}{(q+1)(1-z)} \Big(\frac{z}{q}\Big)^m + \frac{q-z}{(q+1)(1-z)}\Big(\frac{1}{qz}\Big)^m\,,
\end{equation*}
for a given complex number $z \in \mathbb{C} \backslash \{1\}$ (the expression for $\pi_1$ is different and the reader should check \cite[Example 10.8]{schl palma} for the correct definition, but we will not need it here). Kaliszewski, Landstad and Quigg \cite[Example 10.8]{schl palma} have also determined that the characters $\pi_z$ which extend to $^*$-representations of $L^1(G, \Gamma)$ are precisely those with $z \in [-q, -1/q] \cup[1/q, q]$.

We will now consider the $^*$-representation $\pi_{-q}$ of $L^1(G, \Gamma)$ and show that $\pi_{-q}(f^* f) < 0$. First we notice that
\begin{eqnarray*}
 \pi_{-q}(Px_1P)  & = & \frac{1-q(-q)}{(q+1)(1-(-q))} \Big(\frac{-q}{q}\Big) + \frac{q-(-q)}{(q+1)(1-(-q))}\Big(\frac{1}{q(-q)}\Big)\\
& = & - \frac{1+q^2}{(q+1)^2} - \frac{2}{(q+1)^2q}\\
& = & - \frac{q^3 + q +2}{(q+1)^2q}\,.
\end{eqnarray*}
Hence we get
\begin{eqnarray*}
 \pi_{-q}(f^*f) & = & \pi_{-q}\big( (q+1) \,P + (q^2 +q)\,Px_1P \big)\\
 & = & q+1 - (q^2 +q)\frac{q^3 + q +2}{(q+1)^2q}\\
 & = & q+1 - \frac{q^3 + q +2}{q+1}\,.
\end{eqnarray*}
To prove that $\pi_{-q}(f^*f) < 0$ is then equivalent to show that $(q+1)^2 < q^3 + q +2$, or equivalently, $0 < q^3 - q^2 - q +1$, for any prime number $q$. This follows from an elementary calculus argument as follows: letting $F(x) = x^3-x^2-x+1$, we have that $F''(x) = 6x - 2$ is always greater than $0$ for $x \geq 2$ (the first prime number). Hence, $F'(x) = 3x^2-2x-1$ is growing for $x \geq 2$. Since $F'(2) > 0$, it follows that $F'(x)$ is always greater than $0$ for $x \geq 2$. Thus, $F(x)$ is growing in this interval, and since $F(2) >0$, it follows that $F(q) > 0$, for any prime $q$.

Since $\pi_{-q}(f^*f) < 0$  it then follows that not all representations of $L^1(G, \Gamma)$ are $\langle \rangle_R$-positive and consequently $C^*(L^1(G, \Gamma)) \ncong PC^*(G)P$. \qed
\end{proof}

As a particular consequence of the above theorem, it follows that $PSL_2(\mathbb{Q}_q)$ does not have a quasi-symmetric group algebra. Also, together with Hall's result \cite[Proposition 2.21]{hall palma} and the fact that $PSL_2(\mathbb{Q}_q)$ is not amenable, we can say that for this Hecke pair $C^*(G, \Gamma)$ does not exist and $C^*(L^1(G, \Gamma)) \ncong pC^*(\overline{G})p \ncong C^*_r(G, \Gamma)$.

As we have seen in this chapter, the isomorphism $C^*(L^1(G, \Gamma)) \cong pC^*(\overline{G})p$ holds whenever $G$ ( $G_r$ or $\overline{G}$) has subexponential growth. We would like know if the same is true or if one counter-example can be found for the class of amenable groups:

\begin{question}
If $\overline{G}$ is amenable does it follow that $C^*(L^1(G, \Gamma)) \cong pC^*(\overline{G})p$?
\end{question}

\end{document}